\documentclass{amsart}
\usepackage[english]{babel}
\usepackage[cp1250]{inputenc}
\usepackage{amssymb,amsthm,amsfonts}
\usepackage{verbatim}
\usepackage{url}

\usepackage[dvips]{graphicx}
\usepackage{amsmath,amssymb,amsfonts, amscd}
\usepackage{pictexwd,dcpic}
\usepackage{amsthm}
\usepackage{color}
\usepackage[T1]{fontenc}
\usepackage{verbatim}
\usepackage{epstopdf}
\usepackage{enumerate}
\usepackage{url}
\usepackage[toc,page]{appendix}

\usepackage{xcolor}

\theoremstyle{theorem}
\newtheorem{theorem}{Theorem}[section]
\newtheorem{lemma}[theorem]{Lemma}
\newtheorem{corollary}[theorem]{Corollary}
\newtheorem{proposition}[theorem]{Proposition}
\theoremstyle{definition}
\newtheorem{definition}{Definition}
\newtheorem{remark}{Remark}
\newtheorem{example}{Example}

\renewcommand{\AA}{\mathbb A}
\newcommand{\LL}{\mathbb L}
\newcommand{\RR}{\mathbb R}
\newcommand{\NN}{\mathbb N}
\newcommand{\ZZ}{\mathbb Z}

\newcommand{\QQ}{\mathbb Q}

\newcommand{\PP}{\mathbb P}

\newcommand{\spec}{\operatorname{Spec}}
\newcommand{\sgn}{\operatorname{sgn}}
\newcommand{\HH}{\operatorname{HH}}

\newcommand{\lan}{\langle}
\newcommand{\ran}{\rangle}

\newcommand{\D}{\operatorname{D}}
\newcommand{\lam}{\lambda}
\newcommand{\coker}{\operatorname{Coker}}
\newcommand{\cone}{\operatorname{Cone}}
\newcommand{\rk}{\operatorname{rk}}
\newcommand{\pic}{\operatorname{Pic}}
\newcommand{\har}{\operatorname{Har}}
\newcommand{\e}{\operatorname{Ext}}
\newcommand{\h}{\operatorname{Hom}}

\newcommand{\tot}{\operatorname{tot}}

\renewcommand{\sp}{\operatorname{Span}}

\newcommand{\cA}{\mathcal{A}}

\newcommand{\cO}{\mathcal{O}}

\newcommand{\cS}{\mathcal{S}}

\newcommand{\ffi}{\varphi}

\newcommand{\kss}{\scriptscriptstyle}
\newcommand{\kbb}{{\kss \bullet}}

\begin{document}

\title[Hochschild cohomology and deformation quantization]{Hochschild cohomology and deformation quantization of affine toric varieties}
\author{Matej Filip}
\address{Institut f\"ur Mathematik, Freie Universit\"at Berlin, Berlin, Germany}
\email{filip@math.fu-berlin.de}

\begin{abstract}

For an affine toric variety $\spec(A)$, we give a convex geometric description of the Hodge decomposition of its Hochschild cohomology. Under certain assumptions we compute the dimensions of the Hodge summands $T^1_{(i)}(A)$, generalizing the existing results about the Andr\'e-Quillen cohomology group $T^1_{(1)}(A)$. We prove that every Poisson structure on a possibly singular affine toric variety can be quantized in the sense of deformation quantization.  

\end{abstract}

\keywords{Deformation quantization; Hochschild cohomology; Toric singularities}

\subjclass[2010]{13D03, 13D10, 14B05, 14B07, 14M25, 53D55}

\maketitle

\section{Introduction}
The concept of deformation quantization has been appearing in the literature for many years and was established by Bayen, Flato, Fr\o nsdal, Lichnerowicz and Sternheimer in \cite{velik}. 
A major result, concerning the existence of deformation quantization is Kontsevich's formality theorem \cite[Theorem 4.6.2]{kon} which implies that every Poisson structure on a real manifold can be quantized, i.e., admits a star product. 
Kontsevich also extended the notion of deformation quantization into the algebro-geometric setting \cite{kon2}. From Yekutieli's results \cite{yek}, \cite{yek2} it follows that on a 
smooth algebraic variety $X$ (under certain cohomological restrictions)  
every Poisson structure admits a star product. As in Kontsevich's case, the construction is canonical and induces a bijection between the set of formal Poisson structures modulo gauge equivalence and the set of star products modulo gauge equivalence (see also Van den Bergh's paper \cite{vdb}).

When $X=\spec(A)$ is a smooth affine variety, we have the following formality theorem: there exists an $L_{\infty}$-quasi-isomorphism between the Hochschild differential graded Lie algebra  $C^{\kbb}(A)[1]$ and the formal differential graded Lie algebra $H^{\kbb}(A)[1]$ (i.e., the graded Lie algebra $H^{\kbb}(A)[1]$ with trivial differential), extending the Hochschild-Kostant-Rosenberg quasi-isomorphism of the above complexes. Dolgushev, Tamarkin and Tsygan \cite{dol-tam-tsy} proved even a stronger statement by showing that the Hochschild complex $C^{\kbb}(A)$ is formal as a homotopy Gerstenhaber algebra. Consequently, every Poisson structure on a smooth affine variety can be quantized.

Studying non-commutative deformations (also called quantizations) of toric varieties is important for constructing and enumerating noncommutative instantons (see \cite{cir1}, \cite{cir2}), which is closely related to the computation of Donaldson-Thomas invariants on toric threefolds (see \cite{iqb}, \cite{cir}).

In the paper we drop the smoothness assumption and consider the deformation quantization problem for possibly singular affine toric varieties.
In the singular case the Hochschild-Konstant-Rosenberg map is no longer a quasi-isomorphism and thus also the $n$-th Hochschild cohomology group is no longer isomorphic to the Hodge summand $H^n_{(n)}(A)\cong \h_A(\Omega^n_{A|k},A)$. Therefore, other components of the Hodge decomposition come into play, making the problem of deformation quantization interesting from the cohomological point of view.  
In general many parts of the Hodge decomposition are still unknown. The case of complete intersections has been settled in \cite{fro-kon}, where Fr\o nsdal and Kontsevich also motivated the problem of deformation quantization on singular varieties. 
In the toric case Altmann and Sletsj\o e \cite{klaus} computed the Harrison parts of the Hodge decomposition.

Deformation quantization of singular Poisson algebras does not exist in general; see Mathieu \cite{mat}  for counterexamples.
For known results about quantizing singular Poisson algebras we refer the reader to  
\cite{sch} and references therein. 
The associative deformation theory for complex analytic spaces was developed by Palamodov in \cite{pal2} and \cite{pal3}. For recent developments concerning the problem of deformation quantization in derived geometry, see \cite{cal-pan}.

The paper is organized as follows: in Section 2 and 3 we recall definitions and some techniques for computing  Hochschild cohomology. We compute the Hochschild cohomology of a reduced isolated hypersurface singularity in Proposition \ref{prop red iso}.  Section 4 contains computations of Hochschild cohomology for toric varieties. In Theorem \ref{th import} we give a convex geometric description of the Hodge decomposition of the Hochschild cohomology for affine toric varieties. As an application we explicitly calculate $T^1_{(i)}(A)$ for all $i\in \NN$ in the case of two and  three dimensional affine toric varieties (see Propositions \ref{cor t2}, \ref{t1i}). In higher dimensions we compute $T^1_{(i)}(A)$ for affine cones over smooth toric Fano varieties (see Theorem \ref{th fano}). In Section 5 we prove that every Poisson structure on an affine toric variety can be quantized in the sense of deformation quantization. 

\section{Preliminaries}
Let $k$ be a field of characteristic 0 (in Section 5 we assume additionally that $k$ is algebraically closed) and let $A$ be an associative commutative $k$-algebra. We denote by $\cA$ the category of local Artin  $k$-algebras with the residue field $k$ (with local homomorphisms as morphisms) and by $\cS$ we denote the category of sets.

We consider the following deformation problem: a deformation of $A$ over an Artin ring $B$ is a pair $(A',\pi)$, where $A'$ is a $B$-algebra and $\pi:A'\otimes_Bk\to A$ is an isomorphism of $k$-algebras. Two such deformations $(A',\pi_1)$ and $(A'',\pi_2)$ are \emph{equivalent} if there exists an isomorphism of $B$-algebras $\phi:A'\to A''$ such that it is compatible with $\pi_1$ and $\pi_2$, i.e., such that $\pi_1=\pi_2\circ (\phi\otimes_B k)$.
A functor that encodes this deformation problem is
$$\text{Def}_A:\cA\to \cS$$
$$B\mapsto \{\text{deformations of }A \text{ over }B\}/\sim.$$
It is well-known that the differential graded Lie algebra (dgla for short) that controls this deformation problem is the Hochschild dgla $C^{\kbb}(A)[1]$, where $C^{\kbb}(A)$ is the Hochschild cochain complex, i.e., $C^n(A)$ is the space of $k$-linear maps $f:A^{\otimes n}\to A$ (or $A$-module homomorphisms $A\otimes A^{\otimes n}\to A$) with the differential given by 
$$
\begin{array}{ll}
(df)(a_1\otimes \cdots \otimes a_n):=&a_1f(a_2\otimes \cdots \otimes a_n)+\\
&\sum^{n-1}_{i=1}(-1)^{i}f(a_1\otimes \cdots \otimes a_ia_{i+1}\otimes \cdots \otimes a_n)+\\
&(-1)^nf(a_1\otimes \cdots \otimes a_{n-1})a_n.
\end{array}
$$
The $n$-th cohomology groups of this complex is called the $n$-th \emph{Hochschild cohomology group}, denoted by $\HH^n(A)$.
The Lie bracket on $C^{\kbb}(A)[1]$ is coming from the \emph{Gerstenhaber bracket} $[f,g]$ of $f\in C^m(A)$, $g\in C^n(A)$, which is defined as 
$$[f,g]:=f\circ g-(-1)^{(m+1)(n+1)}g\circ f\in C^{m+n-1}(A),$$
where 
$$(f\circ g)(a_1\otimes \cdots \otimes a_{m+n-1}):=$$
$$\sum_{i=1}^m(-1)^{(i-1)(n+1)}f(a_1\otimes \cdots \otimes a_{i-1}\otimes g(a_i\otimes \cdots \otimes a_{i+n-1})\otimes a_{i+n}\otimes \cdots \otimes a_{m+n-1}).$$
The Gerstenhaber bracket equips $C^{\kbb}(A)[1]$ with the structure of a dgla.
 
Gerstenhaber and Schack described the Hodge decomposition of the Hochschild (co-)homology that we will briefly recall (see \cite{ger-sch} for more details). 
In the group ring of the permutation group $S_n$ one defines the shuffle $s_{i,n-i}$ to be $\sum (\sgn \pi)\pi$, where the sum is taken over those permutations $\pi\in S_n$ such that $\pi (1) <\pi (2)<\cdots<\pi (i)$ and $\pi(i+1)<\pi(i+2)<\cdots <\pi(n)$. 
Let $s_n=\sum_{i=1}^{n-1}s_{i,n-i}$.
There exists orthogonal idempotents $e_n(i)\in S_n[\QQ]$ for $i=1,...,n$, whose sum is the unit element. Moreover, for $\lam_i=2^i-2$ it holds that $$s_n=\lam_1e_n(1)+\lam_2e_n(2)+\cdots +\lam_ne_n(n),$$ which gives subcomplexes $C^{\kbb}_{(i)}(A)$, with $C^n_{(i)}(A)=\{f\in C^n(A)~|~f\circ s_n=(2^{i}-2)f\}$.
We have
$$\HH^n(A)\cong H^n_{(1)}(A)\oplus \cdots \oplus H^n_{(n)}(A),$$
where $H^n_{(i)}(A)$  is the $n$-th cohomology of $C^{\kbb}_{(i)}(A)$ (the part of $\HH^n(A)$ corresponding to $e_n(i)$). 

It holds that $H_n^{(n)}(A)\cong \Omega_{A|k}^n$, the $n$-th exterior power of the module of Kähler differentials. If $A$ is smooth, we have $\HH^n(A)\cong H_{(n)}^{n}(A)\cong \h_A(\Omega_{A|k}^n,A)$.

\begin{definition}
The complex $C^{\kbb}_{(1)}(A)$ is called the \emph{Harrison complex} and we will write $\har^{n}(A):=H^n_{(1)}(A)$ for the Harrison cohomology groups.
\end{definition}

\begin{definition}
A skew-symmetric Hochschild $2$-cocycle $p$ that satisfies the Jacobi identity
$$p(a\otimes p(b\otimes c))+p(b\otimes p(c\otimes a))+p(c\otimes p(a\otimes b))=0$$
is called an (algebraic) \emph{Poisson structure} (or a \emph{Poisson bracket}). A commutative algebra
together with a Poisson bracket that also satisfies Leibniz's law is called a \emph{Poisson algebra}. Its spectrum is called an \emph{affine Poisson variety}.
\end{definition}

Using the Hodge decomposition we can equivalently define the Poisson structure as an element $p\in H^2_{(2)}(A)$ with $e_3(3)[p,p]=0$, where $e_3(3)$ is the orthogonal idempotent projecting $C^3(A)$ on $C^3_{(3)}(A)$ 
(see e.g.\ \cite{pal2}).

\begin{definition}
\emph{A one-parameter formal deformation} of $A$ is an associative algebra $(A[[\hslash]],*)$, such that 
$$a*b=ab (\bmod \hslash),$$
for each $a,b\in A$.
We require that $*$ is associative, $k[[\hslash]]$-bilinear and continuous, which means that 
$$\big(\sum_{m\geq 0}b_m\hslash^m\big)*\big(\sum_{n\geq 0}c_n\hslash^n\big)=\sum_{m,n\geq 0}(b_m*c_n)\hslash^{m+n}.$$
\end{definition}

\begin{definition}\label{def quant}
We say that a Poisson structure $p\in H^2_{(2)}(A)$ can be quantized if there exist $\gamma_2$, $\gamma_3$,... in $C^2(A)$, such that 
$$a*b:=ab+\frac{1}{2}p(a\otimes b)\hslash+\gamma_2(a\otimes b)\hslash^2+\gamma_3(a\otimes b)\hslash^3+\cdots $$
is a one-parameter formal deformation.
\end{definition}

Note that when $\text{Har}^3(A)=0$, every Poisson structure can be extended to a second order deformation (i.e.\ $\gamma_2$ always exists (mod $\hslash^3$) since $e_3(3)[p,p]=e_2(3)[p,p]=0$). 

Now we recall the standard notation in the toric setting from \cite{klaus}.
Let $M,N$ be mutually dual, finitely generated, free Abelian groups. We denote by $M_{\RR}$, $N_{\RR}$ the associated real vector spaces obtained via base change with $\RR$. Given a rational, polyhedral cone $\sigma=\lan a_1,...,a_N \ran\subset N_{\RR}$ with apex in $0$ and with $a_1,...,a_N\in N$ denoting its primitive fundamental generators (i.e.\ none of the $a_j$ is a proper multiple of an element of $N$). We define the dual cone $\sigma^{\vee}:=\{r\in M_{\RR}~|~\lan \sigma,r \ran\geq 0\}\subset M_{\RR}$ and denote by $\Lambda:=\sigma^{\vee}\cap M$ the resulting semi-group of lattice points. Its spectrum $\spec(k[\Lambda])$ is called an \emph{affine toric variety}. For $\lam\in \Lambda$ we denote by $x^{\lam}$ the monomial corresponding to $\lam$. Since $\Lambda$ is saturated, 
$\spec(k[\Lambda])$ is normal (see e.g.\ \cite[Theorem 1.3.5]{cox}).
\begin{definition}
A variety $X$ is called \emph{$\QQ$-Gorenstein} if the double dual of some tensor product of $\omega_X$ is an invertible sheaf on $X$.
\end{definition}

The following facts about toric $\QQ$-Gorenstein varieties can be found in \cite[Section 6.1]{alt4}. 
For an affine toric variety given by the cone $\sigma=\lan a_1,...,a_N\ran$ we have that $X$ is $\QQ$-Gorenstein if and only if there exists a primitive element $R^*\in M$ and a natural number $g\in \NN$ such that $\lan a_j,R^*\ran=g$ for each $j=1,...,N$. $X$ is Gorenstein if and only if additionally  $g=1$. In particular, toric $\QQ$-Gorenstein singularities are obtained by putting a lattice polytope $P\subset \AA$ into the affine hyperplane $\AA\times \{g\}\subset N_{\RR}:=\AA\times \RR$ and defining $\sigma:=\text{Cone}(P)$, the cone over $P$. Then the canonical degree $R^*$ equals $(\underline{0},1)$.

\section{Andr\'e-Quillen cohomology}
In this section we recall the geometric approach (using the cotangent complex) for computing the Hochschild cohomology. As an application we compute the Hochschild cohomology of a reduced isolated hypersurface singularity, which will give a more complete view on the results that we will obtain in the next section (see Example \ref{ex 3}).

We will briefly recall the construction of the cotangent complex (for more details see \cite{man}).

\begin{definition}
A dg-algebra $R$ with differential $s$ is called \emph{semifree} if:
\begin{itemize}
\item The underlying graded algebra is a polynomial algebra $k[x_i~|~i\in I]$, where the degree of $x_i$ may vary.
\item There exists a filtration $$\emptyset=I(0)\subset I(1)\subset \cdots , \cup_{n\in \NN}I(n)=I,$$ such that $s(x_i)\in k[x_j~|~j\in I(n)]$ for every $i\in I(n+1)$.
\end{itemize}
\end{definition}
A $k$-\emph{semifree resolution} of an algebra $A$ is a surjective quasi-isomorphism $R\to A$, where $R$ is a semifree $k$-dg-algebra.
Note that a $k$-semifree resolution always exists. The corresponding complex of the $A$-dg module $\Omega_{R|k}\otimes_RA$ 
gives us the element $\LL_{A|k}$ in the derived category $\D(\mathfrak{Mod}_A)$. We call $\LL_{A|k}$ \emph{the cotangent complex}. It is independent of the choice of the $k$-semifree resolution.

We have a quasi-isomorphism between $(\LL_{A|k})[1]$ and $C_{\kbb}^{(1)}(A)$ (see e.g.\ \cite[Proposition 4.5.13]{lod}). Moreover, the \emph{derived exterior powers} $\wedge^i\LL_{A|k}$ (see \cite[Section 3.5.4]{lod} for definitions) give us the following proposition.
\begin{proposition}\label{hodge iso}
There exists a quasi-isomorphism between $\wedge^i(\LL_{A|k})[i]$ and $C_{\kbb}^{(i)}(A)$.
\end{proposition}
\begin{proof}
See \cite[Proposition 4.5.13]{lod}.
\end{proof}
\begin{example}\label{ex hype}
Let $X=\spec(A)$ be a reduced hypersurface, where 
$$A=k[x_1,...,x_N]/(f(x_1,...,x_N)).$$ 
Let us denote $S=k[x_1,...,x_N]$.
The $i$-th derived exterior power $\wedge^i\LL_{A|k}$ is isomorphic to the complex
\begin{equation}\label{wedge iA}
0\to A\xrightarrow{\wedge df}\Omega^1_{S|k}\otimes_SA\xrightarrow{\wedge df}\cdots \xrightarrow{\wedge df}\Omega^i_{S|k}\otimes_SA\to 0,
\end{equation}
where $\Omega^i_{S|k}\otimes_SA$ is the degree $0$ term. We can prove \eqref{wedge iA} by first computing the cotangent complex and since it has only two non-zero terms, we can use \cite[Chapter 4]{sai} (see also \cite{ill}) to compute the derived exterior powers.
\end{example}

\begin{definition}
The $n$-th homology group of $\wedge^i\LL_{A|k}$ is called the $n$-th (higher) \emph{Andr\'e-Quillen homology} group and denoted by $T^{(i)}_n(A)$. The $n$-th cohomology group of $\h_A(\wedge^i\LL_{A|k},A)$ is called the $n$-th (higher) \emph{Andr\'e-Quillen cohomology} group and denoted by $T^n_{(i)}(A)$. 
\end{definition}

In particular, from Proposition \ref{hodge iso} we have an isomorphism of groups 
$$T^{n-1}_{(1)}(A)\cong H^n_{(1)}(A) = \har^n(A),$$
or more generally
$T^{n-i}_{(i)}(A)\cong H^n_{(i)}(A),$
for each $i=1,...,n$. For a smooth algebra $A$ we have 
$$\HH^n(A)\cong H^n_{(n)}(A)\cong T^0_{(n)}(A)$$
and thus we see that for $j>0$ the modules $T^j_{(i)}(A)$ (and similarly for $T_j^{(i)}(A)$) have support on the singular locus.

The next result relates Andr\'e-Quillen cohomology groups with $\e$ groups.

\begin{lemma}\label{lem kun}
Let $X=\spec(A)$ be smooth in codimension $d$. For each $i\geq 1$ and $0\leq j\leq d+1$, we have
$T^j_{(i)}(A)\cong \e_A^j(\Omega^i_{A|k},A)$.
\end{lemma}
\begin{proof}
Since each term of $\wedge^{i}\LL_{A|k}$ is a projective $A$-module for each $i\geq 1$, 
we have a Künneth spectral sequence:
 $$E_2^{p,q}=\e_A^p(T_q^{(i)}(A),A)\Rightarrow T^{p+q}_{(i)}(A).$$
 
The modules $T_q^{(i)}(A)$ have support on the singular locus for $q\geq 1$. Since $A$ is smooth in codimension $d$, we have $\e_A^p(T_q^{(i)}(A),A)=0$ for $q\geq 1$ and $p=0,1,...,d$.
\end{proof}

\begin{proposition}\label{prop red iso}
Let $A$ be a reduced isolated hypersurface singularity in $\AA^N$. 
We have
$$
\HH^n(A)\cong \left\{
\begin{array}{ll}
\h_A(\Omega_{A|k}^n,A)\oplus A/(\frac{\partial f}{\partial x_1},\frac{\partial f}{\partial x_2},...,\frac{\partial f}{\partial x_N})& \text{ if }n<N\\ 
A/(\frac{\partial f}{\partial x_1},\frac{\partial f}{\partial x_2},...,\frac{\partial f}{\partial x_N}) & \text{ if }n\geq N
\end{array}
\right.
$$
\end{proposition}
\begin{proof}
We use Example \ref{ex hype}. The perfect pairing $\Omega^j_{S|k}\otimes_S\Omega^{N-j}_{S|k}\to \Omega^N_{S|k}\cong S$ induces an isomorphism of complexes 
$\h_A(\wedge^N\LL_{A|k},A)[-N]\cong \wedge^N\LL_{A|k}.$
By Michler's result in \cite{mic3} the only nonzero homology groups of $\wedge^N\LL_{A|k}$ are the zeroth and first, both isomorphic to  
$A/(\frac{\partial f}{\partial x_1},\frac{\partial f}{\partial x_2},...,\frac{\partial f}{\partial x_N})$. 
Note that $\Omega^l_{S|k}=0$ holds for $l\geq N+1$ and thus for $i\geq N$ we have that
$$
T^j_{(i)}(A)\cong \left\{
\begin{array}{ll}
A/(\frac{\partial f}{\partial x_1},\frac{\partial f}{\partial x_2},...,\frac{\partial f}{\partial x_N})& \text{ if } j=i-1,i\\
0& \text{ otherwise}.
\end{array}
\right.
$$ 
By \cite{mic3} we also know that $\wedge^k\LL_{A|k}$ is quasi-isomorphic to $\Omega^k_{A|k}$ for $k\leq N-1$. Thus we can easily see that $\e^j_A(\Omega^k_{A|k},A)=0$, if $k\leq N-1$ and $j\ne 0,k-1,k$. Moreover, in the decomposition 
$\e^1_A(\Omega_{A|k}^{n-1},A)\oplus \cdots \oplus \e^{n-1}_A(\Omega_{A|k}^{1},A)$
only one direct summand is nonzero and isomorphic to $\Omega^N_{A|k}\cong A/(\frac{\partial f}{\partial x_1},\frac{\partial f}{\partial x_2},...,\frac{\partial f}{\partial x_N})$. Lemma \ref{lem kun} and
the Hodge decomposition conclude the proof.
\end{proof}

\section{Hochschild cohomology of toric varieties}

From now on we will restrict ourself in the case of toric varieties and try to simplify the results using the lattice grading that comes with toric varieties. The convex geometric description of Harrison cohomology groups of an affine toric variety was given in \cite{klaus}. 
We generalize this result to the case of Hochschild cohomology groups. 

Let $A=\oplus_{i\in \ZZ} A_i$ be a graded $k$-algebra. If $a_0,...,a_p$ are homogenous elements, define the \emph{weight} of $a_0\otimes \cdots \otimes a_p\in A^{\otimes p+1}$ to be $w=\sum|a_i|$, where $|a_i|=j$ means that $a_i\in A_j$. This makes the tensor product $A^{\otimes p+1}$ into a graded $k$-module. 
Since differentials preserve the weight, this equip both $\HH_p(A)$ and $\HH^p(A)$ with the structure of graded $k$-modules.

\subsection{The Hochschild complex in the toric case}\label{subsec 41}
Definitions  and statements in this subsection already appeared in \cite{klaus} for $i=1$. We give a generalization for arbitrary $i\geq 1$. 

In the case when $\spec(A)$ is an affine toric variety there exists $M$-grading on $A$. 
Let $A=k[\Lambda]=k[\sigma^{\vee}\cap M]$. 
\begin{definition}
We say that an element $f\in C^n(A)$ has degree $R\in M$ if $f$ maps an element with weight $w$ to an element of degree $R+w$ in $A$. This means that $f$ is of the form $f(x^{\lambda_1}\otimes \cdots \otimes x^{\lambda_n})=f_0(\lam_1,..,\lam_n)x^{R+\lam_1+\cdots+\lam_n}$. We need to take care that the expression is well defined, i.e., that $f_0(\lam_1,...,\lam_n)=0$ for $R+\lam_1+\cdots \lam_n\not \in \Lambda$ (in the following we will also use $R+\lam_1+\cdots \lam_n\not \geq 0$ since we can look on $M$ as a partially ordered set where positive elements lie in the cone $\Lambda$). Let $C^{n,R}(A)$ denote the degree $R$ elements of $C^n(A)$ and let $C^{n,R}_{(i)}(A)$ denote the degree $R$ elements of $C^{n}_{(i)}(A)$.
\end{definition}

We would like to understand the space $C^{n,R}(A)$ better and the following definition will be useful.

\begin{definition}
$L\subset \Lambda$ is said to be \emph{monoid-like} if for all elements $\lam_1,\lam_2\in L$ the relation $\lam_1-\lam_2\in \Lambda$ implies $\lam_1-\lam_2\in L$. Moreover, a subset $L_0\subset L$ of a monoid-like  set is called \emph{full} if $(L_0+\Lambda)\cap L=L_0.$ 
\end{definition}


For any subset $P\subset \Lambda$ and $n\geq 1$ we introduce $S_n(P):=\{(\lam_1,...,\lam_n)\in P^n~|~\sum_{v=1}^n\lam_v\in P\}$. If $L_0\subset L$ are as in the previous definition, then this gives rise to the following vector spaces ($1\leq i\leq n$):
$$C^n_{(i)}(L,L\setminus L_0;k):=\{\varphi: S_n(L)\to k~|~\ffi\circ s_n=(2^{i}-2)\ffi, \,\varphi \,\text{ vanishes on }\,S_n(L\setminus L_0)\},$$
which turn into a complex with the differential 
$$d^n:C_{(i)}^{n-1}(L,L\setminus L_0;k)\to C_{(i)}^n(L,L\setminus L_0;k),$$
$$(d^n\varphi)(\lam_1,...,\lam_n):=$$ $$\varphi(\lam_2,...,\lam_n)+\sum^{n-1}_{i=1}(-1)^{i}\varphi(\lam_1,...,\lam_i+\lam_{i+1},...,\lam_n)+(-1)^n\varphi(\lam_1,...,\lam_{n-1}).$$

We will see that this complexes will give us a description of a degree $-R\in M$ part of $H^n_{(i)}(A)$.
\begin{definition}
By $H_{(i)}^n(L,L\setminus L_0;k)$ we denote the cohomology groups of the above complex $C_{(i)}^{\kbb}(L,L\setminus L_0;k)$.
\end{definition}

\begin{lemma}\label{25}
$$C^{n,-R}_{(i)}(A)\cong C^{n}_{(i)}(\Lambda,\Lambda\setminus (R+\Lambda);k).$$
\end{lemma}
\begin{proof}
For $f\in C_{(i)}^{n,-R}(A)$, we have $f(x^{\lambda_1}\otimes \cdots \otimes x^{\lambda_n})=f_0(\lam_1,..,\lam_n)x^{\lam_1+\cdots+\lam_n-R}$ and then the isomorphism is given by $f\mapsto f_0$.
\end{proof}

It is a trivial check that Hochschild differentials respect the grading given by the degrees $R\in M$. Thus we get the Hochschild subcomplex $C^{\kbb,-R}_{(i)}$ and we denote the corresponding cohomology groups by $H^{n,-R}_{(i)}(A)\cong T^{n-i,-R}_{(i)}(A)$. 

From definitions it follows that $C^n_{(i)}(A)=\oplus_RC^{n,-R}_{(i)}(A)$, $C^n(A)=\oplus_RC^{n,-R}(A)$ and $H^n_{(i)}(A)=\oplus_RH^{n,-R}_{(i)}(A)$, $\HH^n(A)=\oplus_R\HH^{n,-R}(A)$.

\begin{proposition}\label{prop 1}
Let $R\in M$ and let $A=k[\Lambda]$. We have 
\begin{equation}\label{sisom}
T^{n-i,-R}_{(i)}(A)\cong H_{(i)}^{n}(\Lambda,\Lambda\setminus (R+\Lambda);k).
\end{equation}
\end{proposition}  
\begin{proof}
We use Lemma \ref{25} and the decomposition of the Hochschild cohomology.
\end{proof}

\begin{remark}
We will also use the positive grading $$T^{n-i,R}_{(i)}(A)\cong H_{(i)}^{n}(\Lambda,\Lambda\setminus (-R+\Lambda);k).$$ Poisson structures lie in $T^0_{(2)}(A)$, which is non-zero for positive degrees. 
\end{remark}

\subsection{A double complex of convex sets}\label{subsec42}
In this subsection we follow the paper \cite{klaus} verbatim. Arguments mention in \cite{klaus} in the case $i=1$ work also for arbitrary $i\geq 1$ using the definitions from Subsection \ref{subsec 41}.

Let $\sigma=\lan a_1,...,a_N\ran$.  For $\tau\subset \sigma$ let us define the convex sets introduced in \cite{klaus}:
\begin{equation}\label{eq con set} 
K_{\tau}^R:=\Lambda\cap (R-\mathrm{int}\tau^{\vee}). 
\end{equation}
The above convex sets admit the following properties: 
\begin{itemize}
\item $K_0^R=\Lambda$ and $K_{a_j}^R=\{r\in \Lambda~|~\lan a_j,r \ran<\lan a_j,R \ran\}$ for $j=1,...,N$.
\item For $\tau\ne 0$ the equality $K^R_{\tau}=\cap_{a_j\in \tau}K_{a_j}^R$ holds.
\item $\Lambda \setminus (R+\Lambda)=\cup_{j=1}^NK_{a_j}^R$.
\end{itemize}

We have the following double complexes $C^{\kbb}_{(i)}(K^R_{\kbb};k)$ for each $i\geq 1$. We define 
$C_{(i)}^q(K_{\tau}^R;k):=C_{(i)}^q(K_{\tau}^R,\emptyset;k)$ and
$$C^q_{(i)}(K_p^R;k):=\oplus_{\tau\leq \sigma,\dim \tau=p}C^q_{(i)}(K^R_{\tau};k)~~~~(0\leq p\leq \dim \sigma).$$
The differentials $\delta^p: C_{(i)}^q(K^R_{p})\to C_{(i)}^q(K_{p+1}^R;k)$ are defined in the following way: we are summing (up to a sign) the images of the restriction map $C_{(i)}^q(K_{\tau}^R;k)\to C_{(i)}^q(K_{\tau'}^R;k)$, for any pair $\tau\leq \tau'$ of $p$ and $(p+1)$-dimensional faces, respectively. The sign arises from the comparison of the (pre-fixed) orientations of $\tau$ and $\tau'$ (see also \cite[pg. 580]{cox} for more details).

\begin{example}
The map $\delta: \oplus_{j=1}^NC_{(i)}^q(K_{a_j}^R;k)\to \oplus_{\lan a_j,a_k \ran\leq \sigma}C^q_{(i)}(K_{a_j}^R\cap K_{a_k}^R;k)$ is simply given by: $(f_1,....,f_N)$ gets mapped to $f_j-f_k\in C^q_{(i)}(K_{a_j}^R\cap K_{a_k}^R;k)$. 
\end{example}

The following results (obtained in \cite{klaus} for $i=1$) can also be generalized to $i>1$:

\begin{lemma}\label{acyclic lem}
The canonical $k$-linear map $C_{(i)}^q(\Lambda,\Lambda\setminus (R+\Lambda);k)\to C_{(i)}^q(K^R_{\kbb};k)$ is a quasi-isomorphism, i.e., a resolution of the first vector space.
\end{lemma}
\begin{proof}
For $r\in \Lambda\subset M$ we define the $k$-vector space 
$$V^q_{(i)}(r):=\{\varphi:\{(\lam_1,...,\lam_q)\in \Lambda^q~|~\sum_{v=1}^q\lam_v=r\}\to k~|~\ffi\circ s_n=(2^{i}-2)\ffi \}.$$
and the rest follows as in \cite{klaus}.
\end{proof}

\begin{proposition}\label{total}
$T^{n-i,-R}_{(i)}(A)=H^{n}\big (\tot^{\kbb}(C_{(i)}^{\kbb}(K^R_{\kbb};k))\big )$ for $1\leq i\leq n$.
\end{proposition}
\begin{proof}
We prove the Proposition using first the differentials $d^n$ and Lemma \ref{acyclic lem} and then the differentials $\delta^p$.
\end{proof}

\begin{corollary}\label{second spectral}
Let $i\geq 1$ be a fixed integer. For $q\geq i$ and $p\geq 0$ there is a spectral sequence $$E_1^{p,q}=\oplus_{\tau\leq \sigma,\dim \tau=p} H_{(i)}^q(K_{\tau}^R;k)\Rightarrow T_{(i)}^{p+q-i,-R}(A)=H^{p+q,-R}_{(i)}(A).$$
\end{corollary}
\begin{proof}
We use first the differentials $\delta^p$ and then the differentials $d^n$. 
\end{proof}

\begin{proposition}\label{smooth}
If $\tau\leq \sigma$ is a smooth face, then $H^q_{(i)}(K^R_{\tau};k)=0$ for $q\geq i+1$.
\end{proposition}
\begin{proof}
Let $r(\tau)$ be an arbitrary element of $\text{int}(\sigma^{\vee}\cap \tau^{\perp})\cap M$, i.e., $\tau=\sigma\cap r(\tau)^{\perp}$. We define $R_g:=R-g\cdot r(\tau)$, where $g\in \ZZ$  and we show (with the same idea as in \cite{klaus}) that
\begin{equation}\label{eq smooth 1}
T_{(i)}^{q+\dim \tau-i}(-R_g)=H_{(i)}^{q}(K^R_{\tau};k)\text{ for }g\gg 0,
\end{equation}
for $q\geq i+1$.

Let $T_{(i)}^n(\tau):=T^n_{(i)}(\spec(k[\tau^{\vee}\cap M]))$ and similarly $T_{(i)}^n(\sigma):=T^n_{(i)}(A)$. We have
\begin{equation}\label{eq smooth 2}
T^n_{(i)}(\sigma)\otimes_{k[\sigma^{\vee}\cap M]}k[\sigma^{\vee}\cap M]_{x^{r(\tau)}}=T^n_{(i)}(\tau)=0\text{   for  }n\geq 1,
\end{equation} 
since $k[\tau^{\vee}\cap M]$ equals the localization of $k[\sigma^{\vee}\cap M]$ by the element $x^{r(\tau)}$. The last equality holds because $\tau$ is a smooth face.
From \eqref{eq smooth 2} we see that any element of $T_{(i)}^{q+\dim \tau-i}(-R_g)\subset T_{(i)}^{q+\dim \tau-i}$ will be killed by some power of $x^{r(\tau)}$, which implies that $H_{(i)}^{q}(K^R_{\tau};k)=0$ by \eqref{eq smooth 1}.
\end{proof}

\subsection{The Hochschild cohomology in degree $R\in M$}
The results in this subsection do not follow immediately from \cite{klaus} as in Subsection \ref{subsec42}. Quasi-linear functions (see \cite[Definition 4.1]{klaus}) defined on the convex sets $K^R_{\tau}$ play an important role in describing $T^1_{(1)}(-R)$. In this subsection we show that multi-additive functions (see Definition \ref{def multi-ad}) are the right generalization for describing $T^1_{(i)}(-R)$ for $i\geq 1$.
The main result in this subsection is Theorem \ref{th import}, which is a generalization of \cite[Proposition 5.2]{klaus}. 

We would like to better understand $H^n_{(n)}(K_{\tau}^R;k)$ for $\tau\leq \sigma$. 
These computations are easier then computations for $H^n_{(i)}(K_{\tau}^R;k)$, $i\ne n$, because in the case $i=n$ we do not have coboundaries.

\begin{definition}\label{def multi-ad}
We say that $f\in C^n_{(n)}(L,L\setminus L_0;k)$ is multi-additive if it is additive on every component, provided that the sum of all entries lies in $L$. Being additive in the first component means $f(a+b,\lam_2,...,\lam_n)=f(a,\lam_2,...,\lam_n)+f(b,\lam_2,...,\lam_n)$, with $a+b+\lam_1+\cdots +\lam_n\in L$. 
We denote $$\bar{C}^n_{(n)}(L,L\setminus L_0;k):=\{f\in C^n_{(n)}(L,L\setminus L_0;k)~|~f~\text{is multi-additive}\}.$$
\end{definition}

In the case $n=1$ it holds trivially that $H^{1}_{(1)}(L,L\setminus L_0;k)=\bar{C}^1_{(1)}(L,L\setminus L_0;k)$. Some additional effort is necessary to show this for $n>1$.

\begin{proposition}\label{multi additive}
We have
$$H^{n}_{(n)}(L,L\setminus L_0;k)=\bar{C}^n_{(n)}(L,L\setminus L_0;k)$$
for all $n\geq 1$.
\end{proposition}

\begin{proof}
That every multi-additive function $f\in C^n_{(n)}(L,L\setminus L_0;k)$ satisfies $df=0$ is obvious by definition of $d$. 
For the other direction we use the following computation:
\begin{eqnarray*}\label{lod eq}
&\sum_{\sigma}df(\lam_{\sigma^{-1}(1)},...,\lam_{\sigma^{-1}(n+1)})=\\
&n!\big(f(\lam_1,\lam_3,\lam_4,...,\lam_{n+1})+f(\lam_2,\lam_3,\lam_4,...,\lam_{n+1})-f(\lam_1+\lam_2,\lam_3,\lam_4,...,\lam_{n+1})\big),
\end{eqnarray*}
where the sum is taken over all permutations $\sigma\in S_{n+1}$ such that $\sigma(1)<\sigma(2)$ (similarly as in the proof of Loday \cite[Proposition 1.3.12]{lod}).
\end{proof}

The next Proposition will give us very useful formulas for $H^n_{(n)}(K_{\tau}^R;k)$.

\begin{proposition}\label{prop span}
Let $\tau\leq \sigma$ be a smooth face. The injections $\bar{C}^n_{(n)}(\sp_kK_{\tau}^R;k)\to \bar{C}^n_{(n)}(K_{\tau}^R;k)$ are isomorphisms. Moreover, $\sp_kK_{\tau}^R=\cap_{a_j\in \tau}\sp_kK^R_{a_j},$ and we have 
$$
\sp_kK_{a_j}^R=
\left\{
\begin{array}{ll}
0&\text{ if }\lan a_j,R \ran \leq 0 \\
(a_j)^{\perp} &\text{ if }\lan a_j,R \ran=1\\
M\otimes_{\ZZ}k& \text{ if }\lan a_j,R \ran \geq 2
\end{array}
\right.
$$
\end{proposition}

\begin{proof}
The case $n=1$ was proved in \cite[Proposition 4.2]{klaus}.
We will generalize it to the case $n=2$. The generalization to other $n$ is then immediate.

Let $f\in \bar{C}^2_{(2)}(K_{\tau}^R;k)$. We want to show that $f\in \bar{C}^2_{(2)}(\sp_kK_{\tau}^R;k)$. 
Without loss of generality we can assume that $\tau=\langle a_1,...,a_m \rangle$, with $\langle a_i,R \rangle \geq 2$ for $i=1,...,l$ and $\langle a_j,R \rangle=1$ for $j=l+1,...,m$, since if $R$ was non-positive on any of the generators of $\tau$, then $K_{\tau}^R$ would be empty.

By the smoothness of $\tau$ there exist elements $r_1,...,r_l$ such that $\langle r_i,a_k \rangle=\delta_{ik}$ for $1\leq i \leq l$ and $1\leq k\leq m$. Hence it holds that
$$f(s_{v},s_{w})=\sum_{i=1}^l\sum_{u=1}^l\langle a_{i},s_{v}\rangle \langle a_{u},s_{w}\rangle f(r_{i},r_{u})+f(p_{v},p_{w}),$$
for $s_v,s_w\in K_{\tau}^R$, $p_v:=s_v-\sum_{i=1}^l\langle a_i,s_v \rangle r_i\in \tau^{\perp}\cap M$ and 
$p_w:=s_w-\sum_{i=1}^l\langle a_i,s_w \rangle r_i\in \tau^{\perp}\cap M$.
We can easily show that $\sum_v\sum_wf(s_{v},s_{w})$ does depend only on $s_1:=\sum_vs_v$ and $s_2:=\sum_ws_w$, and not on the summands themselves. Then, $f(s_1,s_2)$ may be defined as this value. The second claim follows as in \cite{klaus} by $\cap_{a_i\in \tau}\sp_kK_{a_i}^R=\cap_{j=l+1}^k(a_j)^{\perp}=\sp_k(\tau^{\perp},r_1,...,r_l)=\sp_kK_{\tau}^R$.
\end{proof}

We write shortly $M_k$ (resp.\ $N_k$) for $M\otimes_{\ZZ}k$ (resp.\ $N\otimes_{\ZZ}k$).

\begin{remark}\label{poisson remark}
Note that $0$ and $1$-dimensional faces are always smooth. For $\tau=0$ we obtain that $\bar{C}^i_{(i)}(\Lambda;k)\cong \bar{C}^i_{(i)}(\sp_k\Lambda;k)\cong \bar{C}^i_{(i)}(M_k;k)$. Thus if $\sigma=\lan a_1,...,a_N\ran\subset M_{k}\cong k^n$, then $f\in \bar{C}^i_{(i)}(\Lambda;k)$ is completely determined by the values $f(s_{k_1},...,s_{k_i})$, for $1\leq k_1<\cdots <k_i\leq n$, where $s_1,...,s_n\in \Lambda$ are linearly independent ($k$-basis in $k^n$). 
\end{remark}

Let $E$ be the minimal set that generates the semigroup $\Lambda:=\sigma^{\vee}\cap M$. We write $E_j^R:=E\cap K^R_{a_j}$, $E_{jk}^R:=E\cap K^R_{a_j}\cap K^R_{a_k}$ for a $2$-face $\lan a_j,a_k\ran\leq \sigma$ and $E_{\tau}^R:=\cap_{a_j\in \tau}E_j^R$ for faces $\tau\leq\sigma$.

\begin{theorem}\label{th import}
Let $X=\spec(A)$ be an affine toric variety that is smooth in codimension $d$. Let $i\geq 1$ be a fixed integer. Then $k$-th cohomology group of the complex 
$$0\to \bar{C}^i_{(i)}(M_k;k)\to \oplus_{j}\bar{C}^i_{(i)}(\sp_kE^R_j;k)\to \cdots \to \oplus_{\tau\leq \sigma,\dim \tau=d+1}\bar{C}^i_{(i)}(\sp_kE^R_{\tau};k)$$
equals $T^{k,-R}_{(i)}(A)$, for $k=0,...,d$ ($\bar{C}^i_{(i)}(M_k;k)$ is the degree $0$ term).

Moreover, if $X$ is an isolated singularity (i.e.\ $\dim(X)=d+1$), then 
$$T^{k,-R}_{(i)}(A)=
\left \{
\begin{array}{ll}
\coker\big(\oplus_{\tau\leq \sigma,\dim \tau=d}\bar{C}^i_{(i)}(K^R_{\tau};k)\to \bar{C}^i_{(i)}(K^R_{\sigma};k)\big) & \text{ if }k=\dim(X)\\
H^{k-\dim(X)+i}_{(i)}(K_{\sigma}^R;k) & \text{ if }k\geq \dim(X)+1
\end{array}
\right.
$$
\end{theorem}
\begin{proof}
By Corollary \ref{second spectral} we have
$$E_1^{p,q}=\oplus_{\tau\leq \sigma,\dim \tau=p} H_{(i)}^q(K_{\tau}^R;k)\Rightarrow T_{(i)}^{p+q-i,-R}(A)= H_{(i)}^{p+q,-R}(A),$$
for $q\geq i$ and $p\geq 0$.
By the assumption $j$-dimensional faces are smooth for $j\leq d$. From Proposition \ref{smooth} it follows that $E_1^{0,q}=E_1^{1,q}=\cdots =E_1^{d,q}=0$, for $q\geq i+1$. Thus $E_2^{p,i}=E_{\infty}^{p,i}=\oplus_{\tau\leq \sigma,\dim \tau=p} H_{(i)}^i(K_{\tau}^R;k)$ for $d+1\geq p\geq 1$.
It follows that $T^{k,-R}_{(i)}(A)$ is the $k$-th cohomology group of the complex 
$$H^i_{(i)}(\Lambda;k)\to \oplus_{j}H^i_{(i)}(K^R_{a_j};k)\to \cdots \to \oplus_{\tau\leq \sigma,\dim \tau=d+1}H^i_{(i)}(K^R_{\tau};k).$$
We conclude the first part using the equality $K^R_{\tau}=\cap_{a_j\in \tau}K_{a_j}^R$ and Proposition \ref{prop span}. 

If $X$ is an isolated singularity then we also have $E_1^{p,q}=0$ for $p\geq d+2$. Thus $E_2^{d+1,q}=E_{\infty}^{d+1,q}=H^{q}_{(i)}(K_{\sigma}^R;k)$ for $q\geq i+1$,  
which finishes the proof.
\end{proof}

\begin{corollary}
Since toric varieties are normal and thus smooth in codimension 1, we obtain that
$T^1_{(i)}(-R)$ equals the cohomology group of the complex
\begin{equation}\label{pom sur}
\bar{C}^i_{(i)}(M_k;k)\to \oplus_{j}\bar{C}^i_{(i)}(\sp_kE^R_j;k)\to \oplus_{\lan a_j,a_k\ran<\sigma}\bar{C}^i_{(i)}(\sp_kE^R_{jk};k).
\end{equation}
\end{corollary}

\subsection{Toric surfaces}
We want to obtain the dimension of $k$-vector spaces $T^{1,-R}_{(i)}(A)$, for all $i\in \NN$, in the case when $A$ is a two-dimensional cyclic quotient singularity (a two-dimensional affine toric variety). Let $X(n,q)$ denote the quotient by the $\mathbb{Z}/n\ZZ$-action $\xi\to \left( {\begin{array}{cc}
  \xi  & 0 \\ 0 & \xi^{q}      \end{array} } \right),$
$(\xi=\sqrt[n]{1})$. $X(n,q)$ is given by the cone $\sigma=\lan a_1,a_2\ran=\lan (1,0),(-q,n)\ran$.
We can develop $\frac{n}{n-q}$ into a continued fraction  $[b_1;b_2,...,b_r]$,
$b_i\geq 2$.
Then $E$ is given as the set $E=\{w^0,...,w^{r+1}\}$, with elements $w^{i}\in \ZZ^2$ and 

\begin{enumerate}
\item $w^0=(0,1)$, $w^1=(1,1)$, $w^{r+1}=(n,q)$,
\item $w^{i-1}+w^{i+1}=b_i\cdot w^{i}$ (i=1,...,r).

\end{enumerate}

Now we compute $T^{1,-R}_{(i)}(A)$ for toric surfaces $A=A(n,q):=k[\lan w^0,w^{r+1}\ran \cap M]$.

\begin{proposition}\label{prop 2}
For $i>2$ we have $\dim T^{1,-R}_{(i)}(A)=0$. Otherwise we have 
$$\dim_kT^{1,-R}_{(i)}(A)=$$ $$\dim_k \bar{C}^i_{(i)}(\sp_kE^R_1;k)+\dim_k\bar{C}^i_{(i)}(\sp_kE^R_2;k)-\dim_k \bar{C}^i_{(i)}(\sp_kE^R_{12};k)-c_i,$$
where 
$$
c_i:=\left\{
\begin{array}{ll}
2=\dim_k \bar{C}^1_{(1)}(M_k;k)&\text{ if }i=1\\
1=\dim_k \bar{C}^2_{(2)}(M_k;k)&\text{ if }i=2 
\end{array}
\right.
$$
\end{proposition}
\begin{proof}
Follows immediately from \eqref{pom sur}, where in this case the last map is surjective.
\end{proof}

\begin{corollary}\label{cor t2}
For $T^1_{(1)}(A)$ we obtain the same results as Pinkham \cite{pin}.
Focusing on $T^{1,-R}_{(2)}(A)$, 
there are four different cases for the multidegree $R\in M\cong \ZZ^2$:

\begin{itemize}
\item $R=w^1$ (or analogously $R=w^r$). We obtain $E_1=\{w^0\}$ and $E_2=\{w^2,...,w^{r+1}\}$. We have $$\dim_k \bar{C}^2_{(2)}(\sp_kE^R_1;k)=\dim_k\bar{C}^2_{(2)}(\sp_kE^R_{12};k)=0$$ and thus Proposition \ref{prop 2} yields $T^{1,-R}_{(2)}(A)=0$. 

\item $R=w^{i}$ $(2\leq i \leq r-1)$. We obtain $E_1=\{w^0,...,w^{i-1}\}$ and $E_2=\{w^{i+1},...,w^{r+1}\}$. We have $\dim_k\bar{C}^2_{(2)}(\sp_kE^R_{12};k)=0$, $$\dim_k\bar{C}^2_{(2)}(\sp_kE^R_{1};k)=\dim_k\bar{C}^2_{(2)}(\sp_kE^R_{2};k)=1$$ and Proposition \ref{prop 2} yields $\dim_kT^{1,-R}_{(2)}(A)=1$.

\item $R=l\cdot w^{i} (1\leq i\leq r, 2\leq l\leq b_i$ for $r\geq 2$, or $i=1,2\leq l\leq b_1$ for $r=1$). We obtain $E_1=\{w^0,...,w^{i}\}$ and $E_2=\{w^{i},...,w^{r+1}\}$.  
We have $\dim_k\bar{C}^2_{(2)}(\sp_kE^R_{12};k)=0$, $$\dim_k\bar{C}^2_{(2)}(\sp_kE^R_{1};k)=\dim_k\bar{C}^2_{(2)}(\sp_kE^R_{2};k)=1$$
and thus Proposition \ref{prop 2} yields $\dim_kT^{1,-R}_{(2)}=1$.

\item For the remaining $R\in M$, either $E_1\subset E_2$ or $E_2\subset E_1$ or $\#(E_1\cap E_2)\geq 2$. In these cases it holds that either $\dim_k\bar{C}^2_{(2)}(\sp_kE^R_{i};k)=0$ for some $i$, or we have $\dim_k\bar{C}^2_{(2)}(\sp_kE^R_{12};k)\ne 0$. Thus in all these cases Proposition \ref{prop 2} yields $\dim_kT^{1,-R}_{(2)}(A)=0$. 

\end{itemize}
\end{corollary}

The following example shows that in the case of Gorenstein toric surfaces ($A_n$-singularities) the computations in this section agree with the computations in the previous section.

\begin{example}\label{ex 3}
Let $A=A(n+1,n)$ be a Gorenstein toric surface, given by the polynomial $f(x,y,z)=xy-z^{n+1}$ in $\AA^3$.
From Proposition \ref{prop red iso} we know that $\HH^3(A)\cong A/(\frac{\partial f}{\partial x},\frac{\partial f}{\partial y},\frac{\partial f}{\partial z})$, which has dimension equal to $n$ (Milnor number of the hypersurface). From Lemma \ref{lem kun} we have $\HH^3(A)\cong \oplus_{i=0}^2\e^i(\Omega^{3-i}_{A|k},A)$ and since $\e^2_{A}(\Omega_{A|k},A)=\h(\Omega^3_{A|k},A)=0$, we see that $\HH^3(A)\cong T^1_{(2)}(A)\cong\e^1(\Omega^2_{A|k},A)$ and thus $\dim_kT^1_{(2)}(A)=n$. Using Corollary \ref{cor t2} we can be even more precise: the cone for $A$ is given by $\sigma=\lan (1,0),(-n,n+1) \ran$. Its continued fraction has $r=1$, $b_1=n+1$ and thus we have $\dim_kT^{1,-R}_{(2)}(A)=1$ for the degrees $R=(2,2),...,(n+1,n+1)$ and  $\dim_kT^{1,-R}_{(2)}(A)=0$ for the other degrees.
\end{example}

\subsection{Higher dimensions} 
Let the cone $\sigma=\lan a_1,...,a_N\ran$ represent an $n$-dimensional toric variety $X_{\sigma}=\spec(A)$, $n\geq 3$. 
For $R\in M$ we define the affine space $$\AA(R):=\{a\in N_{\RR}~|~\langle a,R \rangle=1\}\subset N_{\RR}$$ and consider the polyhedron 
$Q(R):=\sigma\cap \AA(R)\subset \AA(R).$
Vertices of $Q(R)$ are $\bar{a}_j:=a_j/\lan a_j,R  \ran$, for all $j$ satisfying $\langle a_j,R \rangle\geq 1$. We denote $T^1_{(i)}(-R):=T^{1,-R}_{(i)}(A)$.

Altmann  \cite{alt3}, \cite{alt} relates the computation of $T^1_{(1)}(-R)$ with the convex geometry of $Q(R)$ (using Minkowski summands of $Q(R)$). 
We will develop another approach that will also allow us to compute $T^1_{(i)}(-R)$ for $i>1$. 
At the end we will obtain explicit formulas for $3$-dimensional toric varieties (see Proposition \ref{t1i}). As far as we know the techniques that we use to obtain this calculations are new even in the case $i=1$.  In this subsection we also obtain a formula for $T^1_{(i)}(-R)$ for affine cones over smooth toric Fano varieties in arbitrary dimension (see Theorem \ref{th fano}).

The following lemma will be useful.

\begin{lemma}\label{cyc sin}
Let $Y$ be a toric surface given by $\sigma=\lan a_1,a_2 \ran\subset N_{\RR}\cong \RR^2$.
We have 
$\dim_k\sp_kE^R_{12}=\max\{0,W_1(R)+W_2(R)-2-\dim_kT_{(1)}^{1,-R}(Y)\}$,
where 
$$W_j(R):=\left \{
\begin{array}{lr}
2& \text{ if }\lan a_j,R \ran>1\\ 
1& \text{ if }\lan a_j,R \ran=1\\
0& \text{ if }\lan a_j,R \ran\leq 0,
\end{array}
\right.
$$
\end{lemma}
\begin{proof}
It follows immediately by Proposition \ref{prop 2}.
\end{proof}

Let $d_{jk}:=\overline{\bar{a}_j\bar{a}_k}$ denote the compact edges of $Q(R)$ (for $\lan a_j,a_k\ran\leq \sigma$, $\lan a_j,R\ran\geq 1$, $\lan a_k,R\ran\geq 1$). 
We denote the lattice $N\cap\sp_k\lan a_j,a_k\ran$ by $\bar{N}_{jk}$ and its dual with $\bar{M}_{jk}$. Let  $\bar{R}_{jk}$ denote the projection of $R$ to $\bar{M}_{jk}$.

\begin{proposition}\label{t1i}
If the compact part of $Q(R)$ lies in a two-dimensional affine space we have
$$\dim_kT_{(i)}^{1}(-R)=\max\big\{0,\sum_{j=1}^NV^i_j(R)-\sum_{d_{jk}\in Q(R)}Q^i_{jk}(R)-{n\choose i}+s^i_{Q(R)}\big\},$$
where

$$V^i_j(R):=\left \{
\begin{array}{ll}
{n\choose i}& \text{ if }\lan a_j,R \ran>1\\ 
{n-1\choose i}& \text{ if }\lan a_j,R \ran=1\\
0& \text{ if }\lan a_j,R \ran\leq 0,
\end{array}
\right.
$$
$$Q^i_{jk}(R):=\left \{
\begin{array}{ll}
{W_j(R)+W_k(R)+n-4-\dim_kT^{1}_{\lan a_j,a_{k}\ran}(-\bar{R}_{jk})\choose i} & \text{ if }\lan a_j,R \ran,\lan a_k,R \ran\ne 0 \\
0&\text{ otherwise}
\end{array}
\right.
$$
$$s^i_{Q(R)}:=
\left \{
\begin{array}{ll}
\dim_k\wedge^i\big(\bigcap_{d_{jk}\in Q(R)}\sp_kE^R_{jk}\big)& \text{ if }Q(R) \text{ is compact}\\
0&\text{ otherwise}
\end{array}
\right.
$$
\end{proposition}
\begin{proof}

From Theorem \ref{th import} we know that $T^{1}_{(i)}(-R)$ is the cohomology group of the complex 
$$
\bar{C}^i_{(i)}(M_k;k)\to \oplus_{j} \bar{C}^i_{(i)}(\text{Span}_kE_j^R;k) \to \oplus_{\lan a_j,a_k\ran\leq \sigma} \bar{C}^i_{(i)}(\text{Span}_k(E_{jk}^R);k).
$$
Let $f:=(f_1,...,f_N)\in \oplus_j\bar{C}^i_{(i)}(\text{Span}_kE^R_j)$. 
We see that $V^i_j(R)=\dim_k(\wedge^i\text{Span}_kE_j^R)$. 
Assume now that $\text{Span}_kE^R_j$, $\text{Span}_kE^R_{k}\ne \emptyset$, otherwise we have  $\text{Span}_kE^R_{jk}=\emptyset$. 
We can easily verify that $Q^i_{jk}(R)=\dim_k(\wedge^i\text{Span}_kE_{jk}^R)$: 

we have $\dim_k(\text{Span}_kE^R_{jk})=n-2+\dim_k(\text{Span}_k\bar{E}_{jk}^{\bar{R}_{jk}})$, 
where $\bar{E}_{jk}$ is the generating set of $\lan a_j,a_k\ran^{\vee}\cap \bar{M}_{jk}$.
From Lemma \ref{cyc sin} we know that $\dim_k(\text{Span}_k\bar{E}^{\bar{R}}_{jk})=\max\{0,W_j(R)+W_{k}(R)-2-
\dim_kT^{1}_{\lan a_j,a_{k}\ran}(-\bar{R}_{jk})\}.$
Thus we have $$\dim_kT_{(i)}^{1}(-R)=\max\big\{0,\sum_{j=1}^NV^i_j(R)-\sum_{d_{jk}}Q^i_{jk}(R)-{n\choose i}+s^i\},$$
where $s^i$ equals the dimension of the domain of restrictions (that we get with restricting $f_j=f_{k}$ on $\text{Span}_kE_{jk}$) that repeats. We can easily verify that $s^i=s^i_{Q(R)}$.
\end{proof}

Using Proposition \ref{t1i} we can easily compute $T^{1}_{(i)}(-R)$ for three-dimensional affine toric varieties.  From straightforward computation of the formula in Proposition \ref{t1i} we obtain the following corollary.

\begin{corollary}
Let $X$ be an isolated $3$-dimensional toric singularity. Without loss of generality we can assume that generators $a_1,...,a_N$ are arranged in a cycle.
We have the following formulas:
$$
\begin{array}{l}
\dim_kT_{(1)}^1(-R)=\\
=\left \{
\begin{array}{ll}
\max\big\{0,\#\{\bar{a}_j~|~\bar{a}_j\in N,\text{ i.e.\ }\lan a_j,R \ran=1\}-3\big\} & \text{if } R>0\\
\max\big\{0,\#\{\bar{a}_j~|~\bar{a}_j\in N,\text{ not contained in a non-compact edge}\big\} & \text{if } R\not>0,
\end{array}
\right.
\\
\dim_kT^1_{(2)}(-R)=\left \{
\begin{array}{ll}
\max\big\{0,\#\{\bar{a}_j~|~\bar{a}^j\in N\}+C(R)-3\big\} &   \text{if } R>0\\
\max\big\{0,\#\{\bar{a}_j~|~\bar{a}^j\in N\}+C(R)-2\big\} & \text{if  } R\not>0,
\end{array}
\right.
\\
\dim_kT^1_{(3)}(-R)=\max\{0,C(R)-1\} 
\\
\dim_kT^1_{(i)}(-R)=0 \text{ for }i\geq 4,
\end{array}
$$
where $C(R):=\#\{\text{chambers with }\lan a_j,R\ran>1\}$ and a chamber with $\lan a_j,R\ran>1$ means $\lan a_j,R\ran>1$ for $j=j_0,j_0+1,...,j_0+k$  for some $j_0,k\in \NN$ and $\lan a_j,R\ran\not>1$ for $j=j_0-1$ and $j=j_0+k+1$. 
\end{corollary}
\begin{proof}
We use Theorem \ref{t1i} with $n=3$. We also have $T^1_{\lan a_j,a_{j+1}\ran}(-\bar{R}_{j,j+1})=0$ for all $j$ since $X$ is smooth in codimension 2. Let $m_1$ be a number of $a_j$ with $\lan a_j,R \ran=1$ (i.e.\ $m_1$ is the number of lattice vertices of the polytope $Q(R)$) and $m_2$ be a number of vertices $a_j$ with $\lan a_j,R \ran>1$. 

If $R>0$ we have $N=m_1+m_2$ and thus we can easily compute that
$$
s^i_{Q(R)}=\dim_k\wedge^i\bigcap_{j}\text{Span}_kE^R_{j,j+1}={\max\{0,3-m_1\}\choose i}.
$$

For $i=1$ we have $\sum_{j=1}^NV^1_j(R)=3m_2+2m_1$, $\sum_{j=1}^NW_j(R)=2m_1+m_2$ and thus $\sum_{d_{j}}Q^1_{j,j+1}(R)=2\sum_{j=1}^N(W_j(R))-N=4m_2+2m_1-m_1-m_2=3m_2+m_1$. 
Thus we see that $T^1_{(1)}(-R)=\max\{0,m_1-3\}$.

For $i=2$ we have 
$$Q^2_{j,j+1}(R)=\left \{
\begin{array}{ll}
3 & \text{ if }V^2_j(R)=V^2_{j+1}(R)=3 \\
1 & \text{ if }V^2_j(R)=2,V^2_{j+1}(R)=3\text{ or }V^2_{j}(R)=3,V^2_{j+1}(R)=2\\
0&\text{ otherwise}
\end{array}
\right.
$$
and thus
$$V^2_j(R)-Q^2_{j,j+1}(R)=\left \{
\begin{array}{ll}
1 & \text{ if }\lan a_j,R \ran=1 \text{ and }\lan a_{j+1},R \ran=1\\
0 & \text{ if }\lan a_j,R \ran=1 \text{ and }\lan a_{j+1},R \ran=2\\
2 & \text{ if }\lan a_j,R \ran=2 \text{ and }\lan a_{j+1},R \ran=1\\
0 & \text{ if }\lan a_j,R \ran=2 \text{ and }\lan a_{j+1},R \ran=2\\
0&\text{ otherwise}
\end{array}
\right.
$$
from which we easily obtain the formula that we want. 
For $i=3$ we have $\sum_{j=1}^NV^3_j(R)=m_2$, 
$$Q^3_{j,j+1}(R)=\left \{
\begin{array}{ll}
1 & \text{ if }V^3_j(R)=V^3_{j+1}(R)=3 \\
0&\text{ otherwise}
\end{array}
\right.
$$
and the formula follows.

If $R\not>0$ we do not have any compact $2$-faces in $Q(R)$. The only nontrivial case is when we have two vertices that lie on the unbounded edges. We skip this computations since they are similar as in the case $R>0$. 
\end{proof}

\begin{remark}\label{rem res}
When $Q(R)$ is not contained in a two-dimensional affine space, we can still follow 
 the proof of Proposition \ref{t1i} and we obtain that
\begin{equation}\label{eq res}
\dim_kT_{(i)}^{1}(-R)\geq \sum_{j=1}^NV^i_j(R)-\sum_{d_{jk}\in Q(R)}Q^i_{jk}(R)-{n\choose i}.
\end{equation}
The cycles in $Q(R)$ give us some repetitions on the restrictions ($f_j=f_{k}$ on $\text{Span}_kE^R_{jk}$) and thus it is hard to obtain a formula for $\dim_kT_{(i)}^{1}(-R)$ in higher dimensions.  
For every tree $T$ in $Q(R)$ we obtain also upper bounds:
\begin{equation}\label{eq res 2}
\dim_kT_{(i)}^{1}(-R)\leq \sum_{j=1}^NV^i_j(R)-\sum_{d_{jk}\in T}Q^i_{jk}(R)-{n\choose i},
\end{equation}
since no cycles appear in $T$.
\end{remark}

We focus now on higher dimensional toric varieties. We will analyse the case of
 $\QQ$-Gorenstein toric varieties that are smooth in codimension two.

\begin{lemma}\label{lem qgor0}
Let $Y$ be a $\QQ$-Gorenstein variety which is smooth in codimension two. If $R\in M$ is a degree such that $\lan a_j,R\ran\geq 2$ for some $j\in \{1,...,N\}$, then $T^1_{(i)}(-R)=0$ for all $i\geq 1$.
\end{lemma}
\begin{proof}
The hyperplane $H:=\{a\in N_{\RR}~|~\lan a,gR-R^*\ran=0\}$ subdivides the set of generators of $\sigma$: $H^R_{\leq 0}:=\{a_j~|~\lan a_j,R \ran\leq 0\}$, $H^R_1=\{a_j~|~\lan a_j,R \ran= 1\}$ and $H^R_{\geq 2}=\{a_j~|~\lan a_j,R \ran\geq 2\}$.
We fix a vertex $\bar{a}_{j_0}$ of $Q(R)$ with $\lan a_{j_0},R\ran\geq 2$. 
Skipping some of the edges, we can arrange $Q(R)$ into a tree $T$ with the main vertex $\bar{a}_{j_0}$, the set of leaves equal to $H_1^R$ and the set of inner vertices equal to $H^R_{\geq 2}\setminus \bar{a}_{j_0}$. 
From the equation \eqref{eq res 2} we see that $\dim_kT_{(i)}^{1}(-R)\leq \sum_{j=1}^NV^i_j(R)-\sum_{d_{jk}\in T}Q^i_{jk}(R)-{n\choose i}$ and we can easily verify that this is $\leq 0$.
\end{proof}

Deformation theory of affine varieties is closely related to the Hodge theory of smooth projective varieties. We will use the following recent result.
\begin{theorem}\label{tri g}
Let $X=\spec(A)$ be an affine cone over a projective variety $Y$. On $T^q_{(i)}(A)$ we have a natural $\ZZ$ grading and
 if $Y$ is arithmetically Cohen-Macaulay and $\omega_Y\cong \cO_Y(m)$, then 
$$T^q_{(i)}(A)_m=
\left \{
\begin{array}{ll}
H^{n-i,q}_{\text{prim}}(Y) & \text{ if }i>q \\
H^{n-q-1,i}_{\text{prim}}(Y) &\text{ if }i\leq q,
\end{array}
\right.
$$
where $T^q_{(i)}(A)_m$ denotes the degree $m\in \ZZ$ elements of $T^q_{(i)}(A)$ and $H^{p,q}_{\text{prim}}(Y)$ is the primitive cohomology, namely the kernel of the Lefschetz maps $$H^{p,q}(Y)\to H^{p+1,q+1}(Y).$$
\end{theorem}
\begin{proof}
See \cite[Corollary 3.14]{car-fat}.
\end{proof}

We will apply Theorem \ref{tri g} to the case of Fano toric varieties, where reflexive polytopes come into the play. 

\begin{definition}
A full dimensional lattice polytope $P\subset M_{\RR}$ is called \emph{reflexive} if $0\in \text{int}(P)$ and, moreover, its dual 
$$P^{\vee}:=\{a\in N_{\RR}~|~\lan a,P\ran\geq -1\}$$
is also a lattice polytope. Here the expression $\lan a,P\ran$ means the minimum over the set $\{\lan a,r\ran~|~r\in P\}$.
\end{definition}

Reflexive polytopes lead to interesting toric varieties that are important for mirror symmetry.
There is a one-to-one correspondence between Gorenstein toric Fano varieties and reflexive polytopes (see \cite[Theorem 8.3.4]{cox}).

If $X$ is a Gorenstein affine toric variety given by $\sigma=\cone(P)$, where $P$ is a reflexive polytope, then $X$ is an affine cone over a smooth Fano toric variety $Y$, embedded in some $\PP^n$ by the anticanonical line bundle. 

\begin{theorem}\label{th fano}
Let $X=\spec(A)$ be an $n$-dimensional affine cone over a smooth toric Fano variety $Y$ ($n\geq 3$). Then $T^1_{(i)}(A)=0$ for $n\geq 4$ and  $i=2,...,n-2$. Moreover, $\dim_kT^1_{(n-1)}(A)=N-n$ and $T^1_{(k)}(A)=0$ for $k\geq n\geq 3$. Furthermore, $\dim_kT^1_{(1)}(A)=N-3$ for $n=3$ and $T^1_{(1)}(A)=0$ for $n>3$.
\end{theorem}
\begin{proof}
It holds that $H^{p,q}(Y)=0$ for $p\ne q$ (see e.g.\ \cite{bri}) and thus also $H_{\text{prim}}^{p,q}(Y)=0$. 
By Theorem \ref{tri g} we have $T^1_{(i)}(A)_{-1}=0$ for $n\geq 4$ and $i=2,...,n-2$. 
Following the proof of Lemma \ref{lem qgor0}, we see that if $R\ne R^*=(\underline{0},1)$ we have the following options:
\begin{enumerate}
\item there exists $a_j$, such that $\lan a_j,R\ran\geq 2$, which implies that $T^{1,-R}_{(i)}(A)=0$ for all $i\geq 1$ by Lemma \ref{lem qgor0}. 
\item  $H^R_{\geq 2}=0$ and $H_1^R=\{a_j\in F\}$ for a facet $F$. There exists $s\in M$ such that $\lan s,a_j\ran=0$ for all $a_j\in F$. 
If $T^{1,-R}_{(i)}(A)\ne 0$ for some $i$, then $\dim_kT^{1,-R+\alpha s}_{(i)}(A)\ne 0$ for infinitely many $\alpha\in \ZZ$. Thus $\dim_kT^{1}_{(i)}(A)=\infty$, which is a contradiction since $T^{1}_{(i)}(A)$ is supported on the singular locus and $A$ is an isolated singularity. Thus $T^{1,-R}_{(i)}(A)=0$ for all $i\geq 1$.
\item $H^R_{\geq 2}=H^R_1=0$, which trivially implies that $T^{1,-R}_{(i)}(A)=0$.
\end{enumerate} 
Now we focus in the case $i=n-1$. 
Above we saw that $T^{1,-R}_{(n-1)}(A)=0$ if $R\ne R^*$. 
The inequality \eqref{eq res} is in the case $R=R^*$, $i=n-1$ an equality since no restrictions repeat and thus we obtain
$$\dim_kT^{1,-R^*}_{(n-1)}(A)=\max\big\{0,\sum_{j=1}^NV^{n-1}_j(R^*)-\sum_{d_{jk}\in Q(R^*)}Q^{n-1}_{jk}(R^*)-{n\choose n-1}\big\}.$$
Since $V_j^{n-1}(R^*)={n-1\choose n-1}=1$ and $Q_{jk}^{n-1}(R^*)={n-2\choose n-1}=0$ we obtain $T^{1,-R^*}_{(n-1)}(A)=N-n$.
With the same procedure we immediately see that 
 $T^1_{(k)}(A)=0$ for $k\geq n$.
 Finally we focus on the case $i=1$. With the same computations as above we see that
 $\dim_kT^1_{(1)}(A)=0$ if $n>3$.
 If $n=3$, then $\dim_kT^1_{(1)}(A)_{-1}=\dim_kT^1_{(1)}(A)$ as above and $T^1_{(1)}(A)=H_{\text{prim}}^{1,1}(Y)$ by Theorem \ref{tri g}. We have $\dim_kH_{\text{prim}}^{1,1}(Y)=N-3$ by \cite[Theorem 9.4.11]{cox} and thus we conclude the proof.
\end{proof}
\begin{remark}
From Theorem \ref{th fano} and Theorem \ref{tri g} it follows that $$\dim_kH^{1,1}_{\text{prim}}(Y)=N-n=\rk(\pic(Y))-1.$$
\end{remark}

For $i=n-2$ we can generalize Theorem \ref{th fano} to the following:

\begin{proposition}
Let $X=\spec(A)$ be $n$-dimensional $\QQ$-Gorenstein variety given by $\sigma=\cone(P)$, where $P$ is a simplicial polytope. Then 
$T^1_{(n-2)}(A)=0$.
\end{proposition}
\begin{proof}
The only non-clear part is when $X$ is Gorenstein and we consider the degree $R=R^*$. 
Again following the proof of Proposition \ref{t1i} we see that
$$\dim_kT^{1,-R^*}_{(n-2)}(A)=\max\big\{0,\sum_{j=1}^NV^{n-2}_j(R^*)-\sum_{d_{jk}\in Q(R^*)}Q^{n-2}_{jk}(R^*)-{n\choose n-2}\big\},$$
since no restrictions repeat. Let $e$ denote the number of edges in $Q(R^*)$. Since 
$V_j^{n-2}(R^*)={n-1\choose n-2}=n-1$ and $Q_{jk}^{n-2}(R^*)={n-2\choose n-2}=1$, we obtain
$\dim_kT^1_{(i)}(-R^*)=\max\{0,N(n-1)-e-n(n-1)/2\}$. For simplicial polytopes it holds that $e\geq N(n-1)-n(n-1)/2$ by the lower bound conjecture proved in \cite{bar} and thus $\dim_kT^1_{(i)}(-R^*)=0$.
\end{proof}

\begin{remark}
For $i=1$ we can generalize Theorem \ref{th fano} to the following: $\QQ$-Gorenstein toric varieties that are smooth in codimension 2 and $\QQ$-factorial (or equivalently simplicial) in codimension 3  are globally rigid (see  \cite{tot} or \cite{alt4} for the affine case).
\end{remark}

\section{Deformation quantization of affine toric varieties}

In this section we prove that every Poisson structure on an affine toric variety can be quantized. We will use the Maurer-Cartan formalism, Kontsevich's formality theorem (or more precisely its corollary \ref{smooth def quan}) and the GIT quotient construction for an affine toric variety $\spec(A)$ without torus factors: we can write  $\spec(A)=\AA^N/\!\!/G$ for some group $G$. This construction works over an algebraically closed field $k$ of characteristic $0$.
The proof of deformation quantization works also in the case of affine toric varieties with torus factors.

\begin{definition}
Let $\mathfrak{g}$ be a differential graded Lie algebra. The \emph{Maurer-Cartan equation} is 
$$d\xi+\frac{1}{2}[\xi,\xi]=0,~~\xi\in \mathfrak{g}^1,$$
where $\mathfrak{g}^1$ denotes the set of degree $1$ elements in $\mathfrak{g}$.
A solution of this equation is called a \emph{Maurer-Cartan (an MC) element}.
\end{definition}

\begin{lemma}\label{1par-maurer}
One parameter formal deformations $(A[[\hslash]],*)$ of $A$ are in bijection with MC elements of a dgla $\mathfrak{g}:=\big(\hslash C^{\kbb}(A)[1]\big)[[\hslash]]$.
\end{lemma}
\begin{proof}
See \cite[Proposition 4.3.1]{sch}.
\end{proof}

\begin{theorem}[Formality theorem \cite{kon}, \cite{dol-tam-tsy}]\label{th defq}
Let $X=\spec(A)$ be a smooth affine variety. There exists an $L_{\infty}$-quasi-isomorphism between the Hochschild dgla $C^{\kbb}(A)[1]$ and the formal dgla $H^{\kbb}(A)[1]$ (i.e.\ the graded Lie algebra $H^{\kbb}(A)[1]$ with trivial differential).
\end{theorem}

\begin{corollary}\label{smooth def quan}
Every Poisson structure $\pi$ on a smooth affine variety $\spec(A)$ can be quantized. 
\end{corollary}

Now we focus to the case of (singular) toric varieties.
Using the lattice grading the Gerstenhaber bracket can be simplified as follows.

\begin{lemma}\label{lem ger pr}
Let $A=k[\Lambda]$, $f(x^{\lam_1},...,x^{\lam_m})=\sum_{i=0}^{p}f_{i}(\lam_1,...,\lam_m)x^{-R_i+\lam_1+\cdots +\lam_m}\in C^m(A)$ and $g(x^{\lam_1},...,x^{\lam_n})=\sum_{j=0}^{r}g_{j}(\lam_1,...,\lam_n)x^{-S_j+\lam_1+\cdots \lam_n}\in C^n(A)$, where $f_{i}\in C^m(\Lambda,\Lambda\setminus (R_i+\Lambda);k)$, for $i=0,..,p$ and $g_{j}\in C^n(\Lambda,\Lambda\setminus (S_j+\Lambda);k)$ for $j=0,...,r$. Then 
$$[f,g](x^{\lam_1},...,x^{\lam_{m+n-1}})=\sum_{i,j}[f_{i},g_{j}]x^{-R_i-S_j+\lam_1+\cdots \lam_{m+n-1}},$$
where
$$[f_{i},g_{j}]:=f_{i}\circ g_{j}-(-1)^{(m+1)(n+1)}g_{j}\circ f_{i}\in C^{m+n-1}(\Lambda,\Lambda\setminus (R_i+S_j+\Lambda);k),$$
where 
$f_{i}\circ g_{j}(\lam_1,...,\lam_{m+n-1}):=$
$$\sum_{u=1}^mk^n_u\cdot f_{i}(\lam_1,...,\lam_{u-1}, -S_j+\lam_u+\cdots +\lam_{u+n-1},\lam_{u+n},..., \lam_{m+n-1})g_{j}(\lam_u,...,\lam_{u+n-1}),$$
where $k^n_u=(-1)^{(u-1)(n+1)}$. 
\end{lemma}
\begin{proof}
It follows from the isomorphism in Lemma \ref{25}. 
\end{proof}

\begin{proposition}\label{toric poisson}
Every Poisson structure $p$ on an affine toric variety $\text{Spec}(k[\Lambda])$ is of the form 
\begin{equation}\label{eq:poisson}
p(x^{\lam_1},x^{\lam_2})=\sum_{i=0}^df_i(\lam_1,\lam_2)x^{R_i+\lam_1+\lam_2},
\end{equation}
where $f_i\in \bar{C}^2_{(2)}(\Lambda,\Lambda\setminus (-R_i+\Lambda);k)$, $R_i\in M$. We call $f_i(\lam_1,\lam_2)x^{R_i+\lam_1+\lam_2}$ the Poisson structure of degree $R_i$ and we call $p$ a Poisson structure of index $(R_0,...,R_d)$.
\end{proposition}
\begin{proof}
A Poisson structure $p$ is an element of $H^2_{(2)}(k[\Lambda])$ such that $e_3(3)[p,p]=0$.
From Proposition \ref{prop 1} and \ref{multi additive} we know that $$H^{n,R}_{(n)}(k[\Lambda])=H^n_{(n)}(\Lambda,\Lambda\setminus (-R+\Lambda);k)=\bar{C}^n_{(n)}(\Lambda,\Lambda\setminus (-R+\Lambda);k),$$  
thus $p$ is of the form \eqref{eq:poisson}, and $e_3(3)[p,p]=0$ gives us additional restrictions on $f_i$, $i=0,..,d$.
\end{proof}

\begin{example}
For every hypersurface given by the polynomial $g(x,y,z)$ in $\AA^3$, we can define a Poisson structure $\pi_g$ on the quotient $k[x,y,z]/g$, namely:
$$\pi_g:=\partial_x(g)\partial_y\wedge \partial_z+\partial_y(g)\partial_z\wedge\partial_x+\partial_z(g)\partial_x\wedge \partial_y,$$
i.e., we contract the differential 1-form $dg$ to $\partial_x\wedge \partial_y\wedge \partial_z$.
Consider the toric surface $A_n$ given by $g(x,y,z)=xy-z^{n+1}$. We would like to express $\pi_g$ in the form \eqref{eq:poisson}. We see that it holds $\pi_g(x,y)=-(n+1)z^n$, $\pi_g(z,x)=x$ and $\pi_g(y,z)=y$. 
In this case $\Lambda$ is generated by $S_1:=(0,1)$, $S_2:=(1,1)$ and $S_3:=(n+1,n)$, with the relation $S_1+S_3=(n+1)S_2$. We would like to find $p$ of the form \eqref{eq:poisson}, such that $p=\pi_g$. 
With a simple computation, we see that $p$ is of degree $-S_2$:
$$p(x^{\lam_1},x^{\lam_2})=f_0(\lam_1,\lam_2)x^{-S_2+\lam_1+\lam_2},$$
where $f_0(S_1,S_3)=-(n+1)$. The function $f_0$ is with this completely determined by skew-symmetry and bi-additivity. 
\end{example}

Let us now briefly recall the GIT quotient construction $\AA^n/\!\!/G$ of an affine toric variety (see e.g.\ \cite[Chapter 5]{cox}).  
Let $X$ be an affine toric variety without torus factors, i.e., given by the full-dimensional cone $\sigma=\lan a_1,...,a_N\ran\subset N_{\RR}$.
We have a short exact sequence
$$0\to M\xrightarrow{g} \ZZ^{\sigma(1)}\to \text{Cl}(X)\to 0,$$
where $\text{Cl}(X)$ is the class group of $X$, $\sigma(1)=N$ is the number of ray generators and $g$ is an injection map $g(R)=\lan R,a_1 \ran e_1+\cdots +\lan R,a_N \ran e_N$, where $e_j$, $j=1,...,N$ is the standard basis for $\ZZ^N$.
We have $X=\AA^n/\!\!/G$, where $G=\h_{\ZZ}(\text{Cl}(X),k^*)$. 
\begin{remark}
In the above GIT quotient construction we need the assumption that $k$ is algebraically closed. Moreover, the construction can be generalized to semi-projective toric varieties, if we take the GIT quotient of $\AA^n\setminus Z$ for some exceptional set $Z$, which is $\emptyset$ in the case of affine toric varieties.
\end{remark}

The map $g$ induce a semi-group isomorphism between $\Lambda\subset M$ and its image $\Lambda^G:=g(\Lambda)$.
This map determines the isomorphism map of $k$-algebras $$G':k[\Lambda]\to k[x_1,...,x_N]^G,$$ with $G'(x^R)=x^{g(R)}:=x_1^{\lan R,a_1 \ran}\cdots x_N^{\lan R,a_N \ran}$. Elements that lie in $\Lambda^G$  are \emph{$G$-invariant elements}. Thus we have $X=\text{Spec}(k[x_1,...,x_N]/\!\!/G)=\text{Spec}(k[x_1,...,x_N]^G)$.

\begin{proposition}\label{prop 444}
For $\lam,R\in M$ it holds that $$\lam\in \cup_{j\in I}K^R_{a_j} \text{ if and only if } g(\lam)\in \cup_{j\in I}K^{g(R)}_{e_j},$$ 

where $I=\{1,...,N\}$ and $K^{g(R)}_{e_j}$ are the convex sets \eqref{eq con set} of the cone $\lan e_1,...,e_N\ran\subset \RR^N$. 
\end{proposition}
\begin{proof}
By the definition of $g$ we know that $\lan g(\lam),e_j \ran=\lan \lam,a_j \ran$ and 
$\lan g(R),e_j \ran=\lan R,a_j \ran$. For
$g(\lam)\in \cup_jK^{g(R)}_{e_j}$ there exists $j$ such that $\lan g(\lam),e_j \ran< \lan g(R),e_j  \ran$ which means that there exists $j$ such that $\lan \lam,a_j \ran< \lan R,a_j  \ran$, which is equivalent to $\lam \in \cup_jK^{R}_{a_j}$. 
\end{proof}

Let $A=k[\sigma^{\vee}\cap M]$ and $X=\spec(A)$ be a toric variety without torus factors. Let  $T_k=\spec(k[\ZZ^k])$ and $A_k=k[\Lambda \times \ZZ^k]$ ($A_0\cong A$). Every affine toric variety is of the form $X_k=\spec(A_k)=X\times T_k$.
Let $Y_k=\AA^N\times T_k=\spec(B_k)$, where $B_k=k[\NN_0^N\times \ZZ^k]$ and $\NN_0$ is the set of natural numbers with $0$. We define lattices $\widetilde{M}:=M\times \ZZ^k$, $\widetilde{N}:=N\times \ZZ^k$
and a map $g':\Lambda\times \ZZ^k\to \NN^N_0\times \ZZ^k$ with 
$$g'(\lam,\mu)=(g(\lam),\mu).$$

\begin{definition}
Let $(V,\{\cdot,\cdot\})$ be an  affine Poisson variety and let $p:V\to W$ be a dominant map, where $W$ is an affine variety. If there exists a Poisson structure $\{\cdot,\cdot\}_W$ on $W$, such that for every $x\in V$,
$$\{F,G\}_W(p(x))=\{\bar{F},\bar{G}\}(x),$$
for all $F,G\in \cO(W)$ and for all extensions $\bar{F},\bar{G}$ of $F\circ p$ and $G\circ p$, we call $\{\cdot,\cdot\}_W$ \emph{a reduced Poisson structure}.
\end{definition}

\begin{proposition}\label{reduced pois}
Every Poisson structure $p$ on $X_k$ can be seen as a reduced Poisson structure $P$ on $Y_k$. 
\end{proposition}
\begin{proof}
From Proposition \ref{toric poisson} we know that every Poisson structure on $X_k$ is of the form
$$p(x^{\lam_1},x^{\lam_2})=\sum_{i=0}^df_{i}(\lam_1,\lam_2)x^{R_i+\lam_1+\lam_2},$$
 where $f_{i}\in \bar{C}^2_{(2)}(\Lambda\times \ZZ^k,(\Lambda\times \ZZ^k)\setminus (-R_i+(\Lambda\times \ZZ^k));k)$, $R_i\in \widetilde{M}$.

Now we construct a Poisson structure $P$ on a smooth affine variety $Y_k$:
$$P(x^{\lam},x^{\mu})=\sum_{i=0}^dF_{i}(\lam,\mu)x^{g'(R_i)+\lam+\mu},$$
where $F_{i}$ has the property that $F_{i}(g'(\lam_1),g'(\lam_2))=f_{i}(\lam_1,\lam_2)$, for each $i$.

STEP 1: 
Functions $F_{i}$ with the above property exist for each $i$: 

We choose $k+n$ linearly independent vectors $s_1,...,s_{k+n}\in \Lambda\times \ZZ^k$ such that $s_1,...,s_k\in 0\times \ZZ^k$ and $s_{k+1},...,s_{k+n}\in \Lambda\times 0$. 
Note also that $f_{i}$ are completely determined by the values $f_{i}(s_j,s_l)$, for $1\leq j<l \leq k+n$ by Remark \ref{poisson remark}. Since $g'$ is injective we can choose $F_{i}\in \bar{C}_{(2)}^2(\NN^N_{0}\times \ZZ^k;k)$, such that $F_{i}(g'(s_j),g'(s_l))=f_{i}(s_j,s_l)$, for $1\leq j<l \leq k+n$.

Let $t_{1},...,t_{N-n}\in \NN_0^N$ be chosen such that $s_{k+1},...,s_{k+n},t_1,...,t_{N-n}$ determine $\RR$-basis of $\RR^N$.
We choose $F_{i}$ such that $F_{i}(t_j,t_l)=0$ for $1\leq j,l\leq N-n$ and $F_{i}(s_j,t_l)=0$ for $j=1,...,k+n$ and $l=1,...,N-n$ (this will be important to prove the Jacobi identity for $P$ in Step 3). We easily see that it holds $F_{i}(g'(\lam_1),g'(\lam_2))=f_{i}(\lam_1,\lam_2)$.

STEP 2: $P$ is well defined:

That $P(x^{\lam_1},x^{\lam_2})$ is well defined it must for each $i$ hold that $F_{i}(\lam,\mu)=0$ for $g'(R)+\lam+\mu\not\geq 0$. We need to check that this agrees with the property $F_{i}(g'(\lam_1),g'(\lam_2))=f_{i}(\lam_1,\lam_2)$: without loss of generality $\lam_1,\lam_2\in \Lambda\times 0$. We have $F_{i}(g(\lam_1),g(\lam_2))=0$ for $g(R)+g(\lam_1)+g(\lam_2)\not \geq 0$ or equivalently for $g(\lam_1+\lam_2)\in \NN^N_0\setminus \NN^N_0(-g(R))=\cup_{j\in I}K^{-g(R)}_{e_j}$, where $I=\{1,...,N\}$. By Proposition \ref{prop 444} this is equivalent to $\lam_1+\lam_2\in \cup_{j\in I}K^{-R}_{a_j}$ and we indeed have $f_i(\lam_1,\lam_2)=0$ for $R+\lam_1+\lam_2\not \geq 0$.

STEP 3: $P$ satisfies the Jacobi identity:

We have $e_3(3)([p,p])(x^{\lam_1},x^{\lam_2},x^{\lam_3})=0$, since $p$ is a Poisson structure.
Using Lemma \ref{lem ger pr} and the equalities $F_{i}(g'(\lam_1),g'(\lam_2))=f_{i}(\lam_1,\lam_2)$ from Step 1, we see that
$e_3(3)([P,P])(x^{g'(\lam_1)},x^{g'(\lam_2)},x^{g'(\lam_3)})=0$.
Since $e_3(3)[P,P]\in H^3_{(3)}(Y_k)$ we can use Proposition \ref{multi additive} and thus from the construction of $F_{i}$ in Step 1 ($F_{i}(t_j,t_l)=0$ and $F_{i}(s_j,t_l)=0$) we immediately see that $e_3(3)[P,P]=0$.  Thus the Jacobi identity is satisfied.
\end{proof}

Let $\mathfrak{g}$ denote the differential graded Lie algebra $\big(\hslash C^{\kbb}(A_k)[1]\big)[[\hslash]]$ and let $\mathfrak{h}$ denote the differential graded Lie algebra $\big(\hslash C^{\kbb}(B_k)[1]\big)[[\hslash]]$.

\begin{proposition}\label{cor poiss}
Let $\gamma(x^{\lam_1},x^{\lam_2}):=\sum_{m\geq 1}\hslash^m\gamma_{m}(x^{\lam_1},x^{\lam_2})\in \mathfrak{h}^1$ be an MC element of a dgla $\mathfrak{h}$, where $\gamma_1$ is a Poisson structure on $Y_k$ of index $(g'(R_0),...,g'(R_d))$.
 Then $\gamma$ induces an MC element $\widetilde{\gamma}(x^{\lam_1},x^{\lam_2}):=\sum_{m\geq 1}\hslash^m\widetilde{\gamma}_{m}(x^{\lam_1},x^{\lam_2})\in \mathfrak{g}^1$ of the dgla $\mathfrak{g}$, where  $\widetilde{\gamma}_1$ is a reduced Poisson structure on $X_k$ of index $(R_0,...,R_d)$. 
\end{proposition}
\begin{proof}
We prove it just for $d=0$ and $k=0$ (i.e.\ for $\gamma_1$ of degree $R_0$ on a toric variety $X=X_0$ without torus factors). The rest follows easily, just the notation is more tedious.

We know that $\gamma_m(x^{\lam_1},x^{\lam_2})=\gamma_{0m}(\lam_1,\lam_2)x^{mg(R)+\lam_1+\lam_2}$, where 
$$\gamma_{0m}\in C^2(\NN^N_{0},\NN^N_{0}\setminus \NN^N_{0}(-mg(R));k).$$ 
We define $\widetilde{\gamma}_{0m}(\lam,\mu):=\gamma_{0m}(g(\lam),g(\mu))$ and $\widetilde{\gamma}:=\sum_{m\geq 1}\hslash^m\widetilde{\gamma}_m(x^{\lam},x^{\mu})$, where $\widetilde{\gamma}_m(x^{\lam},x^{\mu})=\widetilde{\gamma}_{0m}(\lam,\mu)x^{mR+\lam+\mu}$.

First we need to check that $\widetilde{\gamma}(x^{\lam},x^{\mu})=\sum_{m\geq 1}\hslash^m\widetilde{\gamma}_{m}(x^{\lam},x^{\mu})$ is well defined, i.e., if $mR+\lam+\mu\not \geq 0$, then $\gamma_{0m}(g(\lam),g(\mu))=0$. This can be done as in Step 2 of Proposition \ref{reduced pois}.

Looking only at $G$-invariant elements (i.e.\ $\lam=g(\lam')$ and $\mu=g(\mu')$ for some $\lam', \mu'\in \Lambda$) in the MC equation for $\gamma$ and using Lemma \ref{lem ger pr}, we see that the MC equation also holds  for $\widetilde{\gamma}$. 
\end{proof}

\begin{theorem}\label{th def quan}
Every Poisson structure $p$ on an affine toric variety can be quantized.
\end{theorem}
\begin{proof}
As above let $X_k$ denote an arbitrary affine toric variety.
By Proposition \ref{toric poisson},
$p$ is of the form $p(x^{\lam_1},x^{\lam_2})=\sum_{i=0}^df_{i}(\lam_1,\lam_2)x^{R_i+\lam_1+\lam_2}$ for some $R_i\in \Lambda\times \ZZ^k$. 
By the construction in the proof of Proposition \ref{reduced pois} this Poisson structure can be seen as a reduced Poisson structure of $P$ on $Y_k$: 
$$P(x^{\lam},x^{\mu})=\sum_{i=0}^dF_{i}(\lam,\mu)x^{g'(R_i)+\lam+\mu},$$
where the functions $F_{i}$ have the property that $F_{i}(g'(\lam_1),g'(\lam_2))=f_{i}(\lam_1,\lam_2)$.
Since $P$ is a Poisson structure on a smooth affine variety $Y_k$, we know by Corollary \ref{smooth def quan} that $P$ can be quantized. In other words there exists a one parameter deformation and 
by Lemma \ref{1par-maurer} we know that this correspond to an MC element $\gamma(x^{\lam_1},x^{\lam_2}):=\sum_{m\geq 1}\hslash^m\gamma_{m}(x^{\lam_1},x^{\lam_2})\in \mathfrak{h}^1$, where $\gamma_1$ is of index $(g'(R_0),...,g'(R_d))$. By Proposition \ref{cor poiss} we know that this give us an MC element 
$$\widetilde{\gamma}(x^{\lam_1},x^{\lam_2}):=\sum_{m\geq 1}\hslash^m\widetilde{\gamma}_{m}(x^{\lam_1},x^{\lam_2})\in \mathfrak{g}^1,$$ where  $\widetilde{\gamma}_1$ is a reduced Poisson structure on $X_k$ of index $(R_0,...,R_d)$. By the construction we have $\widetilde{\gamma}_1=p$. Using again Lemma \ref{1par-maurer} we see that $p$ can be quantized.
\end{proof}
\section*{Acknowledgements}
This paper is part of my PhD thesis. I would like to thank to my advisor Klaus Altmann, for his constant support and for providing clear answers to my many questions.
I am also grateful to Victor P. Palamodov, Giangiacomo Sanna and Arne B. Sletsj\o e for useful discussions.

\end{document}